\documentclass[12pt]{article} 

\usepackage{amsmath} 
\usepackage{amssymb}
\usepackage{mathtools}
\usepackage{setspace}
\usepackage[none]{hyphenat}
\usepackage{graphicx}
\usepackage{fancyhdr}
\usepackage{enumerate}

\raggedright

\allowdisplaybreaks

\numberwithin{equation}{section}

\newtheorem{theorem}{Theorem}[section]
\newtheorem{lemma}[theorem]{Lemma}
\newtheorem{proposition}[theorem]{Proposition}

\newenvironment{proof}[1][Proof]{\begin{trivlist}
\item[\hskip \labelsep {\bfseries #1}]}{\end{trivlist}}
\newenvironment{definition}[1][Definition]{\begin{trivlist}
\item[\hskip \labelsep {\bfseries #1}]}{\end{trivlist}}

\newcommand{\qed}{\nobreak \ifvmode \relax \else
      \ifdim\lastskip<1.5em \hskip-\lastskip
      \hskip1.5em plus0em minus0.5em \fi \nobreak
      \vrule height0.75em width0.5em depth0.25em\fi}
      
\begin{document}
\title{A Necessary Condition on the Collatz Conjecture} 
\author{Kerry M. Soileau\\\small{kerry at kerrysoileau dot com}} 

\date{October 11, 2023} 
\maketitle

\singlespace
\begin{abstract}
The Collatz conjecture implies that an iterated function sequence under a certain linear operator, beginning with a certain complex valued function, must converge to a certain complex function.
 
\end{abstract}
\doublespace
\smallskip
\singlespace
\noindent \textbf{Keywords} Collatz, iteration.

\doublespacing

\section{Background}
The Collatz Conjecture is named for the mathematician Lothar Collatz, who introduced it in 1937.\cite{O'Connor}
J. Lagarias provided a useful survey of the subject.\cite{Lagarias} P. Erd\H{o}s remarked that ``Mathematics may not be ready for such problems."\cite{Guy}
By 2020, the conjecture had been verified by computer for all starting values up to $2^{68}.$\cite{Barina}
\section{Introduction}
Let $C(m)$ be the Collatz function given by

\[
    C(n)= 
\begin{cases}
	\frac{n}{2},& \text{for } n \text{ even}\\
	3 n+1,& \text{for } n \text{ odd}
\end{cases}
\]
for $n=1,2,3,\dots$
\begin{definition}
	Suppose $\left\{ a_n \right \}_{n=1}^\infty \subset  \mathbb{C}$ with $\sum \limits_{n=1}^\infty |a_n|<\infty.$ Let the operator $L_n$ be given by 
	\begin{equation}
	L_{0}\left(\sum \limits_{n=1}^\infty a_n e^{n i t} \right) \equiv   \sum \limits_{n=1}^\infty a_n e^{{n i t}}, 
	\end{equation}
	\begin{equation}
	L_{1}\left(\sum \limits_{n=1}^\infty a_n e^{n i t} \right) \equiv   \sum \limits_{n=1}^\infty a_n e^{{C(n) i t}}, 
	\end{equation}
\end{definition}
and in general
\begin{equation}
	L_{n+1} \equiv L_1 \circ L_n \text{ for }n=1,2,3,\cdots.
\end{equation}
Note that in particular if $f(t) \equiv \sum \limits_{n=1}^\infty a_n e^{n i t},$ its absolute convergence implies
\begin{multline}
	L_{1}\left(f(t) \right) 
	=\sum \limits_{n=1}^\infty a_n e^{{C(n) i t}}
	=\sum \limits_{m=1}^\infty a_{2 m} e^{C(2 m) i t}+\sum \limits_{m=0}^\infty a_{2 m+1} e^{C(2 m+1) i t}\\ 
	=\sum \limits_{m=1}^\infty a_{2 m} e^{m i t}+\sum \limits_{m=0}^\infty a_{2 m+1} e^{(6m+4) i t}\\
	=\frac{1}{2} \left(f\left(\frac{t}{2}\right)+f\left(\frac{t}{2}+\pi\right)+
	e^{i t}\left(f(3 t)-f(3 t + \pi)\right)
	\right)
\end{multline}
In the following we consider the case in which $a_n=c^n$ for some $c \in \mathbb{C}$ in the open unit disk and $n=1,2,3,\cdots.$
Suppose $f(t) \equiv \sum \limits_{n=0}^\infty c^n e^{i n t}$ with $|c|<1,$ then $f(t)=\frac{1}{1- c e^{i t}}$ and
\begin{multline}
	L_1(f(t))
	=\sum \limits_{n=0}^\infty c^n e^{{C(n) i t}}=1+c e^{4 i t}+c^2 e^{i t}+c^3 e^{10 i t}+c^4 e^{2 i t}+c^5 e^{16 i t}\\
	+c^6 e^{3 i t}+c^7 e^{22 i t}+c^8 e^{4 i t}+c^9 e^{28 i t}+c^{10} e^{5 i t}+c^{11} e^{34 i t}+c^{12} e^{6 i
   t}+c^{13} e^{40 i t} \\
   +c^{14} e^{7 i t}+c^{15} e^{46 i t}+c^{16} e^{8 i t}+c^{17} e^{52 i t}+c^{18} e^{9 i t}+c^{19} e^{58 i t}+c^{20} e^{10 i t}
   +\cdots\\
   =\frac{1}{2} \left(f\left(\frac{t}{2}\right)+f\left(\frac{t}{2}+\pi\right)+
	e^{i t}\left(f(3 t)-f(3 t + \pi)\right)
	\right)\\
	=\frac{c e^{i t} \left(c^3 \left(-e^{6 i t}\right)-c^2 e^{4 i t}+c+e^{3 i t}\right)}{c^4
   e^{7 i t}-c^2 e^{i t} \left(1+e^{5 i t}\right)+1}
\end{multline}
Continuing, we get
\begin{equation}
\begin{split}
L^2(f(t)) \equiv L_1(L_1(f(t)))\\
=\frac{c e^{i t} \left(c^7 \left(-e^{6 i t}\right)-c^6 e^{5 i t}-c^5 e^{4 i t}-c^4 e^{2 i
   t}+c^3+c^2 e^{4 i t}+c e^{3 i t}+e^{i t}\right)}{c^8 e^{7 i t}-c^4 e^{i t}
   \left(1+e^{5 i t}\right)+1}
\end{split}
\end{equation}
and so on.

\section{Averaging Lemma}
\begin{lemma}
\label{averaginglemma}
	Suppose $a_1,a_2,a_3,\cdots$ is a sequence of complex numbers that eventually repeats with period $P>0.$ By this we mean that there exists some integer $K \geqslant 1$ such that $a_{n+P}=a_n$ for all $n\geqslant K,$ and $K$ is the least integer with this property. Then
\begin{equation}
	\lim_{n \to \infty} \frac{1}{n}\sum \limits_{i=1}^n a_i = \frac{1}{P}\sum \limits_{i=K}^{K+P-1} a_i,
\end{equation}
which is the average of the terms in the repeated subsequence.
\end{lemma}
\begin{proof}
	Let $Q=\min \limits_{K \leqslant j \leqslant K+P-1} \left|\sum \limits_{i=K}^j a_i \right|$ and $R=\max \limits_{K \leqslant j \leqslant K+P-1} \left|\sum \limits_{i=K}^j a_i \right|.$ Let $V=\sum \limits_{i=K}^{K+P-1} a_i$ and $W=\sum \limits_{i=1}^{K-1} a_i.$ Clearly $W$ is the sum of the sequence terms appearing before the repetition starts (if any), $V$ is the contribution due to each repetition subsequence, and $Q$ and $R$ bound the running total contribution of a repetition sequence. If we define the running mean $M_n \equiv \frac{1}{n}\sum \limits_{i=1}^n a_i,$ then for $n \geqslant K+P$ we have
 \begin{multline}
 	n M_n
 	=\sum \limits_{i=1}^{K-1} a_i + \sum \limits_{i=K}^{\Bigl\lfloor \frac{n-K+1}{P} \Bigl\rfloor P+K-1} a_i + \sum \limits_{i=\Bigl\lfloor \frac{n-K+1}{P} \Bigl\rfloor P+K}^n a_i\\
 	=W + \Bigl\lfloor \frac{n-K+1}{P} \Bigl\rfloor V+ \sum \limits_{i=\Bigl\lfloor \frac{n-K+1}{P} \Bigl\rfloor P+K}^n a_i
 \end{multline}
Note that
\begin{equation}
	Q \leqslant \left|\sum \limits_{i=\Bigl\lfloor \frac{n-K+1}{P} \Bigl\rfloor P+K}^n a_i \right| \leqslant R,
\end{equation}
whence
\begin{equation}
 	 Q \leqslant \left|n M_n -W - \Bigl\lfloor \frac{n-K+1}{P} \Bigl\rfloor V \right| \leqslant R,
\end{equation}
and thus
\begin{equation}
 	 \frac{Q}{n} \leqslant \left|M_n - \frac{W}{n} - \frac{1}{n}\Bigl\lfloor \frac{n-K+1}{P} \Bigl\rfloor V \right| \leqslant \frac{R}{n}.
\end{equation}
Taking limits as $n \to \infty,$ we get 
\begin{equation}
 	 0 \leqslant \lim_{n \to \infty}\left|M_n - \frac{W}{n} - \frac{1}{n}\Bigl\lfloor \frac{n-K+1}{P} \Bigl\rfloor V \right| \leqslant 0,
\end{equation}
i.e.
\begin{equation}
	\lim_{n \to \infty}\left|M_n - \frac{W}{n} - \frac{1}{n}\Bigl\lfloor \frac{n-K+1}{P} \Bigl\rfloor V \right|=0.
\end{equation}
Thus
\begin{equation}
	\lim_{n \to \infty}\left|M_n - 0  - \frac{1}{P}\sum \limits_{i=K}^{K+P-1} a_i \right|=0,
\end{equation}
whence
\begin{equation}
	\lim_{n \to \infty} M_n = \frac{1}{P}\sum \limits_{i=K}^{K+P-1} a_i.
\end{equation}
\end{proof}

\section{Convergence Lemma}
\begin{lemma}
\label{convergencelemma}
If for all $n=1,2,3,\cdots$ and $j=1,2,3,\cdots,$
\begin{enumerate}
    \item $a_{n,j} \in \mathbb{C}$,
    \item $\lim \limits_{n \to \infty} a_{n,j} =0$ and
    \item there exist some $K, \rho>0$ such that $|a_{n,j}| \leqslant K \rho^{\,j}<\infty,$
\end{enumerate}
then 
$$\sum \limits_{j=1}^\infty a_{n,j} z^j$$ converges to $0$ for $|z|<\frac{1}{\rho}.$
\end{lemma}
\begin{proof}
	We define $f_n(z) \equiv \sum \limits_{j=1}^\infty a_{n,j} z^j.$ Note that $\inf \limits_k \sup \limits_{j \geqslant k} |a_{n,j}|^{\frac{1}{j}} \leqslant \inf \limits_k \sup \limits_{j \geqslant k} (K \rho^{\,j})^{\frac{1}{j}} =\inf \limits_k \sup \limits_{j \geqslant k} (K^{\frac{1}{j}} \rho)=\rho,$ so the radius of convergence of each $\sum \limits_{j=1}^\infty a_{n,j} z^j$ is at least $\frac{1}{\rho},$ and  thus $f_n(z)$ is well-defined for $|z|<\frac{1}{\rho}.$
Fix $|z_0|<\frac{1}{\rho}.$ We claim that $\lim \limits_{n \to \infty} f_n(z_0)=0.$ Indeed, fix $\epsilon>0.$ For each $m=1,2,3,\cdots,$ we define $A_{n,m}(z_0)=\sum \limits_{j=1}^m a_{n,j} z_0^j$ and $B_{n,m}(z_0)=\sum \limits_{j=m+1}^\infty a_{n,j} z_0^j.$ Clearly $f_n(z_0)=A_{n,m}(z_0)+B_{n,m}(z_0).$
Let $m_0$ be the smallest positive integer satisfying 
$m_0 >  \frac
{\ln \frac{\epsilon }{2 K}+\ln (1-\rho |z_0|)}
{\ln \rho+\ln |z_0|}
-1 .$ 
Then $|B_{n,m_0}(z_0)|=\Big|\sum \limits_{j=m_0+1}^\infty a_{n,j} z_0^j\Big|\leqslant \sum \limits_{j=m_0+1}^\infty |a_{n,j}| |z_0|^j\leqslant \sum \limits_{j=m_0+1}^\infty K \rho^{\,j} |z_0|^j=\frac{K |\rho z_0|^{m_0+1}}{1-\rho |z_0|}<\frac{\epsilon}{2},$ the last inequality due to our assumption on $m_0.$ Thus $|B_{n,m_0}(z_0)|<\frac{\epsilon}{2}$ for every $n=1,2,3,\cdots.$ Next, since $\lim \limits_{n \to \infty} a_{n,j}=0$ there exist integers $N_j>0$ such that for each $j=1,2,\cdots,m_0$ we have $|a_{n,j}|<\frac{\epsilon}{2(m_0+1)}$ for $n \geqslant N_j.$ By taking $N=\max \limits_{1 \leqslant j \leqslant m_0} N_j,$ we have found an $N$ such that  $n \geqslant N$ implies $|a_{n,j}|<\frac{\epsilon}{2(m_0+1)}$ for $j=1,2,\cdots,m_0.$ Now note that for such $n$ we have $|A_{n,m_0}(z_0)|=\Big|\sum \limits_{j=1}^{m_0} a_{n,j} z_0^j\Big| \leqslant \sum \limits_{j=1}^{m_0} |a_{n,j}| |z_0|^j < \sum \limits_{j=1}^{m_0} |a_{n,j}| < \sum \limits_{j=1}^{m_0} \frac{\epsilon}{2(m_0+1)}=\frac{\epsilon}{2}.$
Finally, $|f_n(z_0)|=|A_{n,m_0}(z_0)+B_{n,m_0}(z_0)| \leqslant |A_{n,m_0}(z_0)|+|B_{n,m_0}(z_0)|<\frac{\epsilon}{2}.$ This implies $\lim \limits_{n \to \infty} \sum \limits_{j=1}^\infty a_{n,j} z^j=0$ for any $|z|<1,$ as desired.
\end{proof}

\section{Necessary Condition}
Let $|c|<1.$ We define
\begin{equation}
	f_n(\theta) \equiv L_1(f_{n-1}(\theta))
\end{equation}
for $n=1,2,\cdots,$ where $f_0(\theta) \equiv 1+c e^{i \theta}+c^2 e^{2 i \theta}+ \cdots = \frac{1}{1-c\, e^{i \theta}}.$
This implies
\begin{equation}
  f_n(\theta)=\sum \limits_{r=1}^\infty c^r e^{C^n(r) i \theta}
\end{equation}
where $C^0(m) \equiv m$ and $C^k(m) \equiv C(C^{k-1}(m))$ for $m=1,2,3,\cdots.$
Next, we now define the $n^{th}$ mean:
\begin{equation}
  M_n(\theta)=\frac{1}{n} \sum \limits_{j=1}^{n} f_j(\theta)
\end{equation}
for $n=1,2,3,\cdots.$
Finally we define
\begin{equation}
  g_{n,m}(\theta) \equiv \frac{1}{n} \sum \limits_{j=1}^{n} e^{C^j(m)i\theta}
\end{equation}
for $n=1,2,3,\cdots$ and $m=1,2,3,\cdots.$ 
We then have 
\begin{equation}
  M_n(\theta)=\sum \limits_{m=1}^\infty c^m g_{n,m}(\theta)
\end{equation}
\begin{proposition}
If every Collatz sequence eventually reaches $1,$
	then for each $m=1,2,3,\cdots,$ 
\begin{equation}
	\lim \limits_{n \to \infty} g_{n,m}(\theta)=\frac{1}{3}\left( e^{i\theta}+e^{2 i\theta}+e^{4 i\theta} \right).
\end{equation}
\end{proposition}
\begin{proof}
	Fix $m \in \{1,2,3,\cdots\}$ and $\theta \in [0,2\pi).$ Then the sequence 
	\begin{equation}
		\left\{ e^{C^n(m)i\theta} \right\}_{n=0}^\infty
	\end{equation}
\end{proof}
eventually repeats with period $3$ and repeating subsequence 
$\left\{ e^{4i\theta},e^{2i\theta},e^{i\theta} \right\}.$ By Lemma \ref{averaginglemma}, it follows that $\lim \limits_{n \to \infty} g_{n,m}(\theta)=\frac{1}{3}\left( e^{i\theta}+e^{2 i\theta}+e^{4 i\theta} \right).$
\begin{theorem}
\label{maintheorem}
	If every Collatz sequence eventually reaches $1,$ then for each $\theta \in [0,2 \pi)$ we have 
\begin{equation}
	\lim \limits_{n \to \infty} M_n(\theta)=1+\frac{1}{3}\left( e^{i\theta}+e^{2 i\theta}+e^{4 i\theta} \right) \frac{c}{1-c}
\end{equation}
\end{theorem}
\begin{proof}
Let 
\begin{equation}
\lim \limits_{n \to \infty} M_k(\theta)=\lim \limits_{n \to \infty}\sum \limits_{m=1}^\infty c^m g_{n,m}(\theta).
\end{equation}
Note that 
\begin{multline}
\left| g_{n,m}(\theta)- \frac{1}{3}\left( e^{i\theta}+e^{2 i\theta}+e^{4 i\theta} \right)\right|\\
=\left| \frac{1}{n} \sum \limits_{j=1}^{n} e^{C^j(m)i\theta}-\frac{1}{3}\left( e^{i\theta}+e^{2 i\theta}+e^{4 i\theta} \right)\right|\\
\leqslant \left| \frac{1}{n} \sum \limits_{j=1}^{n} e^{C^j(m)i\theta}\right|+\left|-\frac{1}{3}\left( e^{i\theta}+e^{2 i\theta}+e^{4 i\theta} \right)\right|\\
\leqslant \frac{1}{n} \sum \limits_{j=1}^{n} \left|e^{C^j(m)i\theta}\right|+\frac{1}{3}\left| e^{i\theta}+e^{2 i\theta}+e^{4 i\theta} \right|\\
\leqslant 1 +1=2<\infty\\
\end{multline}
for all $m=1,2,3,\cdots$ and $m=1,2,3\cdots.$

Now we invoke Lemma \ref{convergencelemma} to infer that
\begin{equation}
	\lim \limits_{n \to \infty}\sum \limits_{m=1}^\infty c^m \left( g_{n,m}(\theta)- \frac{1}{3}\left( e^{i\theta}+e^{2 i\theta}+e^{4 i\theta} \right) \right)=0,
\end{equation}
i.e.
\begin{multline}
	\lim \limits_{n \to \infty}\sum \limits_{m=1}^\infty c^m g_{n,m}(\theta)
	=\sum \limits_{m=1}^\infty c^m \frac{1}{3}\left( e^{i\theta}+e^{2 i\theta}+e^{4 i\theta} \right)\\
	=\frac{1}{3}\left( e^{i\theta}+e^{2 i\theta}+e^{4 i\theta} \right) \frac{c}{1-c},
\end{multline}
whence
\begin{equation}
	\lim \limits_{n \to \infty} M_n(\theta)
=\frac{1}{3}\left( e^{i\theta}+e^{2 i\theta}+e^{4 i\theta}  \right) \frac{c}{1-c},	
\end{equation}
as desired.
\end{proof}

\end{document}